\documentclass[11pt]{amsart}
\RequirePackage{amsmath}
\RequirePackage{amsthm}
\RequirePackage{amsfonts}
\RequirePackage{amssymb}
\RequirePackage{multirow}
\RequirePackage{multicol}
\RequirePackage{graphicx}
\RequirePackage{float}
\RequirePackage{enumerate} 
\RequirePackage{cancel}
\RequirePackage{appendix}
\RequirePackage{calrsfs}
\RequirePackage{cite}
\RequirePackage[figurewithin=section]{caption}
\RequirePackage{xcolor}
\RequirePackage{hyperref}
\RequirePackage{mathtools}

\RequirePackage{tikz}
\usetikzlibrary{babel}
\usetikzlibrary{shapes,arrows}
\tikzstyle{block} = [draw, fill=blue!20, rectangle, 
    minimum height=3em, minimum width=3em]
\tikzstyle{sum} = [draw, fill=blue!20, circle, node distance=1cm]
\tikzstyle{input} = [coordinate]
\tikzstyle{output} = [coordinate]

\numberwithin{equation}{section}

\newtheorem{theorem}{Theorem}[section]
\newtheorem{proposition}[theorem]{Proposition}
\newtheorem{corollary}[theorem]{Corollary}
\newtheorem{lemma}[theorem]{Lemma}

\newtheorem{thm}{Theorem}

\theoremstyle{definition}

\newtheorem{definition}[theorem]{Definition}

\newtheorem{remark}[theorem]{Remark}

\newcommand{\C}{\mathbb{C}}

\newcommand{\R}{\mathbb{R}}
\newcommand{\D}{\mathbb{D}}
\newcommand{\N}{\mathbb{N}}

\renewcommand{\H}{\mathbb{H}}

\newcommand{\abs}[1]{\left| #1 \right|}

\title[Extremal rate]{Extremal rate of convergence in\\
continuous dynamics}

\author[F. J. Cruz-Zamorano]{Francisco J. Cruz-Zamorano}
\address{Departamento de An\'alisis Matem\'atico, Facultad de Ciencias, Universidad de La Laguna, Avenida Astrof\'isico Francisco Sanchez S/N, 38206 San Crist\'obal de La Laguna, Santa Cruz de Tenerife, Spain}
\email{fcruzzam@ull.edu.es}

\author[K. Zarvalis]{Konstantinos Zarvalis}
\address{Department of Mathematics, Aristotle University of Thessaloniki, 54124, Thessaloniki, Greece}
\email{zarkonath@math.auth.gr}

\date{}
\thanks{Cruz-Zamorano was supported by Ministerio de Innovaci\'on y Ciencia, Spain, project PID2022-136320NB-I00 and by Ministerio de Universidades, Spain, through the action Ayuda del Programa de Formaci\'on de Profesorado Universitario, reference FPU21/00258.\\
Zarvalis was partially supported by Junta de Andaluc\'{i}a, grant number QUAL21 005 USE}

\newcommand{\condI}{{\normalfont (I)}}
\newcommand{\condII}{{\normalfont (II)}}
\newcommand{\condIIp}{{\normalfont (II')}}
\newcommand{\condIIpp}{{\normalfont (II'')}}

\begin{document}

\begin{abstract}
This paper deals with semigroups of holomorphic self-maps of the upper half-plane that exhibit an extremal (i.e. the slowest possible) rate of convergence to their Denjoy--Wolff point. The main novelty lies in the parabolic case of zero hyperbolic step. We provide several characterizations for such semigroups in terms of the Herglotz representation of their infinitesimal generators, the conformality at the Denjoy--Wolff point of a modification of their associated Koenigs function, and more.
\end{abstract}

\maketitle

\section{Introduction}

A \textit{one-parameter continuous semigroup of holomorphic functions in the upper half-plane} $\H\vcentcolon=\{z\in\C:\mathrm{Im}(z)>0\}$ (from now on a \textit{semigroup in} $\H$ for short) is a family $(\phi_t)$ of holomorphic self-maps $\phi_t\vcentcolon\H\to\H$, $t\ge0$, which satisfy the following three properties:
\begin{enumerate}
    \item[(a)] $\phi_0(z)=z$, for all $z\in\H$;
    \item[(b)] $\phi_{t+s}(z)=\phi_t(\phi_s(z))$, for all $z\in\H$ and all $t,s\ge0$;
    \item[(c)] $\lim_{t\to0^+}\phi_t(z)=z$, for all $z\in\H$.
\end{enumerate}
The theory of such semigroups is elegant and rich, having significant ties with several branches of mathematics, such as Geometric Function Theory and Operator Theory. The modern research with respect to semigroups initiated in 1978 through the work of Berkson and Porta \cite{BP} which reignited the interest in the field. For a profound presentation of the theory of semigroups we refer to the books \cite{Abate,BCDM}.

An immediate consequence of the definition is that every term $\phi_t$, $t\ge0$, of a semigroup is univalent \cite[Theorem 8.1.17]{BCDM} which provides an advantage when dealing with continuous semigroups instead of the discrete semigroups in iteration theory. Given $z\in\H$, the curve $\gamma_z\vcentcolon[0,+\infty)\to\H$ with $\gamma_z(t)=\phi_t(z)$ and the set $\{\phi_t(z)\vcentcolon t\ge0\}$ are interchangeably called the \textit{orbit} of $z$. Their dynamical behaviour is a well-studied topic in Complex Analysis and the main line of investigation for this article. 

The asymptotic behaviour of the orbits is described by the continuous version of the Denjoy--Wolff Theorem \cite[Theorem 8.3.1]{BCDM}. This fundamental result states that for any \textit{non-elliptic} semigroup $(\phi_t)$ (that is $\phi_t(z)\ne z$, for all $z\in\H$ and all $t>0$), all orbits converge as $t\to+\infty$ to a unique boundary point $\tau\in\R\cup\{\infty\}$, called the \textit{Denjoy--Wolff point} of $(\phi_t)$. By conjugating with a suitable automorphism of $\H$ we can always normalize this point to be infinity, an assumption we maintain throughout the course of this work. Depending on the angular derivative of $\phi_1$ at the Denjoy--Wolff point infinity and the asymptotic behaviour of the orbits in conjunction with the hyperbolic distance $d_\H$ of the upper half-plane, non-elliptic semigroups are distinguished in three main types: \textit{hyperbolic}, \textit{parabolic of positive hyperbolic step}, and \textit{parabolic of zero hyperbolic step} (for detailed information about the classification of semigroups see Subsection \ref{subsec:cont-dynamics}). In general, parabolic semigroups of zero hyperbolic step are more elusive and require further attention in their examination due to their more complicated nature.

A key question in the field, which has drawn considerable attention recently, is the rate at which these orbits move away from their starting point and approach the Denjoy--Wolff point, see e.g. \cite{DB-Rate-H,DB-Rate-P,BCDM-Rate,BCZZ,BK-Speed,Bracci,BCK,CGCRP-UnifConv,CZ-FiniteShift,KTZ}. Depending on the type of the semigroup, sharp lower bounds have been discovered with regard to the convergence of $\phi_t(z)$, $z \in \H$, towards its Denjoy--Wolff point, as $t \to +\infty$. For parabolic semigroups of zero hyperbolic step, a seminal result by Betsakos \cite[Theorem 1(c)]{DB-Rate-P} provides a sharp universal lower bound on this rate. Specifically, for any such semigroup with Denjoy--Wolff point at infinity, it holds that $\liminf_{t \to +\infty} (|\phi_t(z)|/\sqrt{t}) > 0$ for every $z \in \H$. When restricting to the other two types of semigroups, similar bounds are true. In particular, instead of $\sqrt{t}$, for hyperbolic semigroups and parabolic semigroups of positive hyperbolic step, the corresponding limit inferior contains $e^{\lambda t}$ and $t$, respectively, where $\lambda>0$ is a number which explicitly depends on the semigroup; more information follows in Subsection \ref{subsec:cont-dynamics} and Section \ref{sec:ext-rate}. 

This paper focuses on semigroups for which the convergence of their orbits towards their Denjoy--Wolff points is as slow as possible, meaning that the rate of convergence matches the established lower bound. We will refer to those as semigroups \textit{of extremal rate}. Our main objective is to establish a comprehensive set of characterizations for parabolic semigroups of zero hyperbolic step that exhibit this extremal rate. A crucial preliminary finding is that if $\limsup_{t \to +\infty} (|\phi_t(z)|/\sqrt{t}) < +\infty$ for some $z\in\H$, then the limit exists and is finite for all $z\in\H$ (see Lemma \ref{lemma:rate-0hs} and Theorem \ref{thm:0HS-Herglotz}). This allows us to formally define a parabolic semigroup $(\phi_t)$ of zero hyperbolic step as being of \textit{extremal rate} if $\lim_{t \to +\infty} (|\phi_t(z)|/\sqrt{t}) \in (0,+\infty)$ for all $z\in\H$.

Following the necessary background in Section \ref{sec:preliminaries}, we formally define the notion of extremality in Section \ref{sec:ext-rate}. There, we also provide the motivation behind our research by briefly reviewing certain known results concerning the hyperbolic case and the parabolic case of positive hyperbolic step. We will see how the extremal rate of convergence for those types of semigroups may be tied to some tools of Geometric Function Theory. Our aim is to extend such results to the parabolic case of zero hyperbolic step and to involve even more tools in this discussion.

Indeed, in Section \ref{sec:Herglotz}, we connect the extremal rate to the Herglotz representation of a self-map of $\H$ (see \eqref{eq:Herglotz}) which plays a central role in the study of semigroups. To be more precise, every semigroup may be considered as the continuous flow of a semi-complete vector field. In particular, given a semigroup $(\phi_t)$ in $\H$ with Denjoy--Wolff point infinity, there exists a holomorphic self-map $G \vcentcolon \H \to \H \cup \R$ such that 
\begin{equation}\label{eq:inf-gen}
\frac{\partial\phi_t(z)}{\partial t}=G(\phi_t(z)), \quad\text{for all }z\in\H\text{ and all }t\ge0.
\end{equation}
For this map we will prove the following (cf. Theorem \ref{thm:0HS-Herglotz}):

\begin{thm}
Let $G \vcentcolon \H \to \H \cup \R$ be a holomorphic map with Herglotz representation given by
\begin{equation*}
G(z) = \alpha z + \beta + \int_{\R}\dfrac{1+sz}{s-z}d\mu(s), \qquad z \in \H.
\end{equation*}
Then, $G$ is the infinitesimal generator of a parabolic semigroup $(\phi_t)$ of zero hyperbolic step in $\H$ with Denjoy--Wolff point infinity and of extremal rate if and only if
\begin{equation*}\alpha = 0, \quad \int_{\R}s^2d\mu(s) < +\infty, \quad\text{and}\quad \beta = \int_{\R}sd\mu(s).
\end{equation*}
\end{thm}
Whether the conditions on the Herglotz representation of $G$ given in the preceding result are satisfied or not can be characterized in terms of the behaviour of $G$ near infinity; see Theorem \ref{thm:0HS-properties}. Other characterizations of dynamical properties using the Herglotz representation have appeared in \cite{CCZRP-Slope,HS,CZ-FiniteShift}.

Next, Section \ref{sec:conformality} explores the link between the extremal rate and the geometric properties of the \textit{Koenigs function} $h$ of a semigroup $(\phi_t)$, which linearizes the dynamics via the equation $h \circ \phi_t = h + t$. Since infinity is the Denjoy--Wolff point of $(\phi_t)$, we have that $\angle\lim_{z \to \infty}h(z) = \infty$ \cite[Proposition 9.4.8]{BCDM}. Therefore, such a map is said to be \textit{conformal at infinity} if the limit $\angle\lim_{w\to\infty}(h(w)/w)$ exists and is a non-zero complex number. Contrary to the other cases (see Theorems \ref{thm:hyp-conformality} and \ref{thm:phs-conformality} or the original references \cite{DB-FiniteShift,CDP-C2}), we notice that the notion characterizing the property of extremal rate is not the conformality of $h$ itself at infinity, but rather that of its square root. Our central finding in this direction is the following (cf. Theorem \ref{thm:0HS}):

\begin{thm}
Let $(\phi_t)$ be a parabolic semigroup of zero hyperbolic step in $\H$, with Denjoy--Wolff point infinity and Koenigs function $h$. Then, $(\phi_t)$ is of extremal rate if and only if there exists a constant $c\in\C$ such that $\sqrt{h-c}$ is conformal at infinity.
\end{thm}
Note that the constant $c$ is needed solely for the purpose of rendering the square root well-defined.

To continue with, we establish a new characterization in Section \ref{sec:hyp-dist} using the asymptotic behaviour of the hyperbolic distance $d_{\H}$ between the orbits and an arbitrary point in $\H$. Our main theorem on that front, which builds upon \cite{Bracci}, is the following (cf. Theorem \ref{thm:hyperbolic rate 0HS}):

\begin{thm}
    Let $(\phi_t)$ be a parabolic semigroup of zero hyperbolic step in $\H$ with Denjoy--Wolff point infinity. Then, $(\phi_t)$ is of extremal rate if and only if the limit
    \begin{equation*}
    \lim\limits_{t\to+\infty}\left(d_\H(i,\phi_t(z))-\frac{1}{4}\log t\right)
    \end{equation*}
    exists in $\R$ for some (and hence all) $z\in\H$.
\end{thm}

Finally, in Section \ref{sec:comp-operators}, we translate our results to the unit disc $\D$ and derive consequences for the norms of the associated semigroup of composition operators acting on classical function spaces (Corollary \ref{thm:norm 0HS}).

A parallel research about the extremal rate of convergence of the discrete orbits of holomorphic self-maps $f \vcentcolon \H \to \H$ towards their Denjoy--Wolff point can be found in \cite{CZZ}. Unlike the continuous setting, no universal lower bound for the rate of convergence of the discrete orbits of parabolic self-maps of zero hyperbolic step is known. Therefore, \cite{CZZ} solely deals with hyperbolic self-maps and parabolic self-maps of positive hyperbolic step, although many of the results given there can be translated to the continuous setting. Thus, the stated results concerning parabolic semigroups of zero hyperbolic step are completely novel.

\section{Preliminaries}
\label{sec:preliminaries}

\subsection{Hyperbolic distance}
During our analysis, we will often use hyperbolic geometry and, more precisely, the hyperbolic distance on simply connected domains. For a comprehensive treatment of the subject, we refer the reader to \cite[Chapter 5]{BCDM}. Our primary setting is the upper half-plane $\H$, where the \textit{hyperbolic distance} between two points $z, w \in \H$ is given by
\begin{equation}
\label{eq:hypdistance in H}
d_{\H}(z,w) = \dfrac{1}{2}\log\left(\dfrac{1+\rho_{\H}(z,w)}{1-\rho_{\H}(z,w)}\right), \quad \rho_{\H}(z,w) =\left| \dfrac{z-w}{z-\overline{w}}\right|, \quad z,w \in \H,
\end{equation}
and $\rho_\H$ is known as the \textit{pseudo-hyperbolic distance} in $\H$.

A fundamental property of the hyperbolic metric is its invariance under conformal transformations. This allows for a consistent definition of hyperbolic distance on any simply connected domain $\Omega \subsetneq \C$. Specifically, if $\phi: \H \to \Omega$ is a conformal mapping, the hyperbolic distance on $\Omega$ is defined as
\begin{equation}\label{eq:hyperbolic invariance}
d_{\Omega}(z,w): = d_{\H}(\phi^{-1}(z),\phi^{-1}(w)), \quad \text{for all } z,w \in \Omega.
\end{equation}
This definition is independent of the choice of $\phi$.

The relationship between holomorphic maps and hyperbolic distance is governed by the Schwarz--Pick Lemma \cite[Corollary 1.1.16]{Abate}. It asserts that any holomorphic map $f: \Omega_1 \to \Omega_2$ between two simply connected domains is non-expansive with respect to their hyperbolic metrics:
$$d_{\Omega_2}(f(z),f(w)) \leq d_{\Omega_1}(z,w), \quad\textup{for all } z,w \in \Omega_1.$$
Equality holds for some $z \neq w$ if and only if $f$ is a conformal mapping from $\Omega_1$ onto $\Omega_2$.

An immediate consequence of the Schwarz--Pick Lemma is the domain monotonicity property of the hyperbolic distance. If $\Omega_1,\Omega_2$ are two simply connected domains with $\Omega_1 \subset \Omega_2 \subsetneq \C$, the inclusion map from $\Omega_1$ into $\Omega_2$ is holomorphic, which implies
\begin{equation}\label{eq:hyperbolic monotonicity}
d_{\Omega_2}(z,w) \leq d_{\Omega_1}(z,w), \quad \textup{for all }z,w\in\Omega_1.
\end{equation}

\subsection{Continuous dynamics}
\label{subsec:cont-dynamics}
Having provided in the Introduction the fundamental notions in the theory of semigroups, we proceed to their detailed classification. Let $(\phi_t)$ be a semigroup in $\H$. As hinted at in the Introduction, if there exist $z\in\H$ and $t>0$ such that $\phi_t(z)=z$, then the semigroup is called \textit{elliptic}. Otherwise, we say that $(\phi_t)$ is \textit{non-elliptic}. For the purposes of our article, we will exclusively work with non-elliptic semigroups. Therefore, from now on, we completely disregard elliptic semigroups.

As mentioned before, up to a conjugation with an automorphism of $\H$, we will consistently assume that the Denjoy--Wolff point of $(\phi_t)$ is $\tau=\infty$. In that case, non-elliptic semigroups are further classified in terms of their angular derivative at infinity. Namely, there exists $\lambda \geq 0$ \cite[Theorem 8.3.1]{BCDM} such that
$$\angle\lim_{z \to \infty}\dfrac{\phi_t(z)}{z} = e^{\lambda t}, \qquad t \geq 0.$$
The value $\lambda$ is called the \textit{spectral value} of $(\phi_t)$. The semigroup $(\phi_t)$ is said to be \textit{hyperbolic} if $\lambda > 0$ and, otherwise, \textit{parabolic}.

Non-elliptic semigroups can be further classified depending on their hyperbolic step. Namely, a semigroup $(\phi_t)$ is said to be of \textit{positive hyperbolic step} whenever
$$\lim_{t \to +\infty}d_{\H}(\phi_{t+1}(z),\phi_t(z)) > 0$$
for some (hence all) $z \in \H$ \cite[Corollary 9.3.3]{BCDM}. Otherwise, the limit is always zero and $(\phi_t)$ is said to be of \textit{zero hyperbolic step}. This new property creates a distinction only within the class of parabolic semigroups, since hyperbolic semigroups are always of positive hyperbolic step \cite[Corollary 9.3.7]{BCDM}.

Furthermore, we bring up one more way to distinguish semigroups. A non-elliptic semigroup $(\phi_t)$ is said to be of \textit{finite shift} if there exists some $z\in\H$ such that $\lim_{t\to+\infty}\mathrm{Im}(\phi_t(z))<+\infty$. Note that due to Julia's Lemma \cite[Theorem 2.1.19]{Abate}, $\mathrm{Im}(\phi_t(z))$ is a non-decreasing function of $t\ge0$ and thus the limit above always exists in $(0,+\infty]$. In addition, if the imaginary part of one orbit remains bounded, then the same is true for all orbits; see \cite[Lemma 17.7.4]{BCDM}. Hence, we say that $(\phi_t)$ is of \textit{infinite shift} in case $\lim_{t\to+\infty}\mathrm{Im}(\phi_t(z))=+\infty$ for some --- and equivalently --- all $z\in\H$. By \cite[Proposition 17.7.3]{BCDM}, we know that a semigroup of finite shift is necessarily parabolic of positive hyperbolic step. As it turns out in the next section, the property of finite shift is inextricably linked to the extremal rate for parabolic semigroups of positive hyperbolic step.

To close the section, we introduce one more prominent tool in the examination of the dynamical behaviour of a semigroup in $\H$, the Koenigs function. More precisely, for any non-elliptic semigroup $(\phi_t)$ in $\H$, there exists an essentially unique univalent map that linearizes the orbits.

In the hyperbolic setting, the existence of the Koenigs function is due to Valiron \cite{Valiron}. We give the following result for further reference.

\begin{theorem}
{\normalfont \cite[Theorem 9.3.5 and Proposition 9.3.10]{BCDM}}
Let $(\phi_t)$ be a hyperbolic semigroup in $\H$ with spectral value $\lambda > 0$. Then, there exists a univalent  map $h \colon \H \to \H$ such that $h \circ \phi_t = e^{\lambda t}h$ for all $t \geq 0$. Moreover, if there exist two univalent maps $h_1, h_2 \colon \H \to \H$ such that $h_i \circ \phi_t = e^{\lambda t}h_i$ for all $t \geq 0$ and $i \in \{1,2\}$, then there exists $\alpha > 0$ such that $h_1 = \alpha h_2$.
\end{theorem}

On the other hand, in the parabolic setting, this linearization may be provided as a solution to Abel's functional equation. We provide all the necessary information in the following result:

\begin{theorem}
\label{thm:Koenigs-phs}
{\normalfont \cite[Theorem 9.3.5 and Proposition 9.3.10]{BCDM}}
Let $(\phi_t)$ be a parabolic semigroup in $\H$. Then, there exists a univalent map $h \colon \H \to \C$ such that $h \circ \phi_t = h + t$ for all $t \geq 0$. Moreover:
\begin{enumerate}[\hspace{0.5cm}\rm(a)]
\item If there exist two univalent maps $h_1, h_2 \colon \H \to \C$ such that $h_i \circ \phi_t = h_i + t$ for all $t \geq 0$ and $i \in \{1,2\}$, then there exists $w \in \C$ such that $h_1 = h_2+w$.
\item $(\phi_t)$ is of positive hyperbolic step if and only $h(\H)$ is contained in a horizontal half-plane.
\end{enumerate}
\end{theorem}

In any case, the univalent map $h$ is called the \textit{Koenigs function} of $(\phi_t)$ and the simply connected domain $\Omega\vcentcolon=h(\H)$ its \textit{Koenigs domain}.

\section{Extremal rate}
\label{sec:ext-rate}
\subsection{Motivation: hyperbolic dynamics}
Let $(\phi_t)$ be a hyperbolic semigroup in $\H$ with Denjoy--Wolff point infinity and spectral value $\lambda > 0$. It is known (see e.g. \cite[Theorem 1(a)]{DB-Rate-P}) that
\begin{equation*}
\liminf_{t \to +\infty}\dfrac{\abs{\phi_t(z)}}{e^{\lambda t}} > 0, \qquad z \in \H.
\end{equation*}
That is, the orbits of a hyperbolic semigroup escape to infinity at least exponentially fast. Actually, arguing as in \cite[Theorem 4.2]{BCDM-Rate}, we see that for every $z \in \H$, the limit
\begin{equation*}
\lim_{t \to +\infty}\dfrac{\phi_t(z)}{e^{\lambda t}}
\end{equation*}
exists in $(\C \setminus \{0\})\cup\{\infty\}$. Moreover, its value is either infinity for all $z \in \H$, or a finite number in $\H$ for all $z \in \H$.

In the spirit of the Introduction, we say that a hyperbolic semigroup $(\phi_t)$ in $\H$ with Denjoy--Wolff point infinity and spectral value $\lambda > 0$ is of \textit{extremal rate} if
$$\lim_{t \to +\infty}\dfrac{\phi_t(z)}{e^{\lambda t}} \in \H$$
for some (and hence all) $z \in \H$. Hyperbolic semigroups of extremal rate were characterized by Betsakos, Contreras, and D\'iaz-Madrigal in terms of the Koenigs function of the semigroup as follows:
\begin{theorem}
{\normalfont \cite[Theorem 4.2]{BCDM-Rate}}
\label{thm:hyp-conformality}
A hyperbolic semigroup $(\phi_t)$ in $\H$ with Denjoy--Wolff point infinity is of extremal rate if and only if its Koenigs function $h \vcentcolon \H \to \H$ is conformal at infinity, that is
\begin{equation*}
\angle\lim_{z \to \infty}\dfrac{h(z)}{z} \in \C \setminus \{0\}.
\end{equation*}
\end{theorem}

\subsection{Motivation: parabolic dynamics of positive hyperbolic step}
Some of the ideas in the previous subsection can be translated to the parabolic setting of positive hyperbolic step. Namely, Betsakos \cite[Theorem 1(b)]{DB-Rate-P} proved that if $(\phi_t)$ is such a semigroup with Denjoy--Wolff point infinity, then
\begin{equation*}
\liminf_{t \to +\infty}\dfrac{\abs{\phi_t(z)}}{t} > 0,
\end{equation*}
for all $z \in \H$. Arguing as in \cite[Proposition 3.4]{CZZ}, one may prove that the limit
\begin{equation*}
\lim_{t \to +\infty}\dfrac{\phi_t(z)}{t}
\end{equation*}
exists in $(\C\setminus \{0\})\cup\{\infty\}$ for all $z \in \H$. Moreover, its value is independent of $z \in \H$, and it is either infinity or a real number. For this reason, we say that a parabolic semigroup $(\phi_t)$ of positive hyperbolic step in $\H$ with Denjoy--Wolff point infinity is of \textit{extremal rate} if
$$\lim_{t \to +\infty}\dfrac{\phi_t(z)}{t} \in \R$$
for some (hence, all) $z \in \H$. Combining previous results in the literature, we may state the following characterization of extremality.
\begin{theorem}
{\normalfont (see \cite[Theorem 3]{DB-FiniteShift}, \cite[Theorem 4.1]{CDP-C2}, and \cite[Proposition 3.4]{CZZ})}
\label{thm:phs-conformality}
Let $(\phi_t)$ be a parabolic semigroup of positive hyperbolic step in $\H$ with Denjoy--Wolff point infinity and Koenigs function $h\vcentcolon\H\to\H$. The following are equivalent:
\begin{enumerate}
\item[\textup{(a)}] $(\phi_t)$ is of extremal rate;
\item[\textup{(b)}] $h$ is conformal at infinity, that is
\begin{equation*}
\angle\lim_{z \to \infty}\dfrac{h(z)}{z} \in \C \setminus \{0\};
\end{equation*}
\item[\textup{(c)}] $(\phi_t)$ is of finite shift.
\end{enumerate}
\end{theorem}
Therefore, we see that the class of parabolic semigroups of positive hyperbolic step that are of extremal rate coincides with the class of non-elliptic semigroup with finite shift.

\subsection{Parabolic dynamics of zero hyperbolic step}
In this work we primarily deal with parabolic semigroups of zero hyperbolic step in $\H$. For such a semigroup $(\phi_t)$, Betsakos \cite[Theorem 1(c)]{DB-Rate-P} proved that
\begin{equation*}
\liminf_{t \to +\infty}\dfrac{\abs{\phi_t(z)}}{\sqrt{t}} > 0,
\end{equation*}
for all $z \in \H$. Moreover, the following fact holds.
\begin{lemma}
\label{lemma:rate-0hs}
{\normalfont (cf. \cite[Corollary 5.2]{BCDM-Rate})}
Let $(\phi_t)$ be a parabolic semigroup of zero hyperbolic step in $\H$ with Denjoy--Wolff point infinity. Assume that there exists $z_0 \in \H$ and $L \in (0,+\infty]$ such that $\lim_{t \to +\infty}\left(\abs{\phi_t(z_0)}/\sqrt{t}\right) = L$. Then, for every $z\in\H$, $\lim_{t \to +\infty}\left(\abs{\phi_t(z)}/\sqrt{t}\right) = L$.
\end{lemma}
Motivated by this, we introduce the following definition.
\begin{definition}\label{def:extremal 0HS}
Let $(\phi_t)$ be a parabolic semigroup of zero hyperbolic step in $\H$ with Denjoy--Wolff point infinity. We say that $(\phi_t)$ is \textit{of extremal rate} if there exists $L \in (0,+\infty)$ such that
\begin{equation}
\label{eq:def-0HS-lim}
\lim_{t \to +\infty}\dfrac{\abs{\phi_t(z)}}{\sqrt{t}} = L
\end{equation}
for some (and hence all) $z \in \H$.
\end{definition}

In contrast to the other cases, it might seem unclear whether the existence of $\lim_{t \to +\infty}\left(\abs{\phi_t(z)}/\sqrt{t}\right)$ in $(0,+\infty)$ also leads to the existence of the limit $\lim_{t \to +\infty}\left(\phi_t(z)/\sqrt{t}\right)$ (i.e. the orbits may converge with a non-trivial set of slopes \cite[Sections 17.5 and 17.6]{BCDM}), or whether the limit in \eqref{eq:def-0HS-lim} can equivalently be substituted by a limit infimum or supremum. Nevertheless, a combination of two results by Bracci (see \cite[Proposition 3.8 and Remark 7.5]{Bracci}) ensures at once that, if the rate of convergence is indeed extremal, then the convergence towards the Denjoy--Wolff point is necessarily non-tangential. On top of that, in Corollary \ref{corollary:0HS-Herglotz}, we will further prove that the convergence is in fact orthogonal. This will aid us in clarifying some of the aforementioned uncertainty.

\section{Herglotz representation}
\label{sec:Herglotz}
In this section, we will characterize parabolic semigroups of zero hyperbolic step and of extremal rate in terms of the Herglotz representation of their infinitesimal generators. As a consequence, we will establish that the orbits of such semigroups necessarily converge to their Denjoy--Wolff points orthogonally.

We commence by reviewing certain pieces of information concerning infinitesimal generators. For extended details, we refer the interested reader to \cite[Chapter 10]{BCDM}. In general, every dynamical property of a semigroup is reflected on its infinitesimal generator. For instance, combining \cite[Theorem 2.6]{BP} with \cite[Corollary 10.1.12]{BCDM} we may deduce the following well-known fact:

\begin{lemma}
\label{lemma:infinitesimal-generator}
Let $(\phi_t)$ be a semigroup in $\H$. Then, $(\phi_t)$ is non-elliptic with Denjoy--Wolff point infinity if and only if its infinitesimal generator $G$ is a holomorphic map from $\H$ into $\H \cup \R$. In such a case, $(\phi_t)$ is parabolic if and only if
\begin{equation*}
\angle\lim_{z \to \infty}\dfrac{G(z)}{z} = 0.
\end{equation*}
\end{lemma}

Thus, a seminal result due to Herglotz \cite[Theorem 6.2.1]{Aaronson} certifies that the infinitesimal generator $G$ of a non-elliptic semigroup $(\phi_t)$ in $\H$ can be uniquely written in the form
\begin{equation}
\label{eq:Herglotz}
G(z) = \alpha z + \beta + \int_{\R}\dfrac{1+sz}{s-z}d\mu(s), \quad z \in \H,
\end{equation}
where $\alpha \geq 0$, $\beta \in \R$ and $\mu$ is a positive finite measure on $\R$. Relation \ref{eq:Herglotz} is known as the \textit{Herglotz representation} of $G$ and a similar equality is valid for any holomorphic self-map of the upper half-plane. In particular, some of the parameters in the representation above may be directly computed from $G$. More specifically,
\begin{equation}
\label{eq:coefficients}
\alpha = \angle\lim_{z \to \infty}\dfrac{G(z)}{z} = G'(\infty), \qquad \beta = \mathrm{Re}(G(i)).
\end{equation}

For the sake of brevity, we introduce the following notation.
\begin{definition}
We say that a holomorphic map $G \vcentcolon \H \to \H \cup \R$ is \textit{represented} by the triplet $(\alpha,\beta,\mu)$ if \eqref{eq:Herglotz} holds.
\end{definition}

In particular, by Lemma \ref{lemma:infinitesimal-generator}, $(\phi_t)$ is parabolic with Denjoy--Wolff point infinity if and only if its infinitesimal $G$ is represented by a triple $(\alpha,\beta,\mu)$ with $\alpha = 0$.

\begin{remark}
If there exists $z \in \H$ and $C \in \R$ such that $G(z) = C$, then $G(z) = C$ for all $z \in \H$. In such a case, two options arise. If $C = 0$, then $\phi_t = \mathrm{Id}_{\H}$ for all $t \geq 0$. Otherwise, $\phi_t(z) = Ct+\phi_0(z)$ for all $t > 0$, and so $(\phi_t)$ is of positive hyperbolic step. In particular, we conclude that the infinitesimal generator of a parabolic semigroup of zero hyperbolic step is a self-map of $\H$.
\end{remark}

In order to ease the exposition, we now provide an explicit sufficient condition for the orthogonal convergence of the orbits of a parabolic semigroup in terms of the Herglotz representation of its infinitesimal generator. Recall that a curve $\gamma:[0,+\infty)\to\H$ converges \textit{orthogonally} to infinity provided that $\lim_{t \to +\infty}\gamma(t) = \infty$ and $\lim_{t\to+\infty}\arg(\gamma(t)) = \pi/2$.
\begin{lemma}
\label{lemma:orthogonal-convergence}
Let $(\phi_t)$ be a parabolic semigroup in $\H$ with Denjoy--Wolff point infinity, and let $G \vcentcolon \H \to \H$ be its infinitesimal generator, represented by the triplet $(\alpha,\beta,\mu)$. Assume that
$$\alpha = 0, \qquad \int_{\R}s^2d\mu(s) < + \infty, \qquad\textup{and} \quad \beta = \int_{\R}sd\mu(s).$$
Then, the orbits of $(\phi_t)$ converge to infinity orthogonally.
\begin{proof}
From the assumptions, we see that $G$ can be written as
$$G(z) = \int_{\R}\dfrac{1+s^2}{s-z}d\mu(s),$$
where $\mu$ is not the null measure. Therefore,
$$iyG(iy) = \int_{\R}\dfrac{(1+s^2)(isy-y^2)}{s^2+y^2}d\mu(s).$$
Letting $y\to+\infty$, Lebesgue's Dominated Convergence Theorem allows us to interchange limit and integration to get
$$\lim_{y \to + \infty}(iyG(iy)) = - \int_{\R}(1+s^2)d\mu(s) \in (-\infty,0).$$
As a consequence, since the curve $y \mapsto iy$ converges to infinity orthogonally (and thus non-tangentially) in $\H$, Lindel\"of's Theorem implies
\begin{equation}
\label{eq:angle-zGz}
\angle\lim_{z \to \infty}(zG(z)) \in (-\infty,0).
\end{equation}
In particular, for $z \in \H$ we can write $G(z) = Q(z)/z$ where $Q$ is holomorphic on $\H$ with
\begin{equation}
\label{eq:orthogonal 1}
\angle\lim_{z \to \infty}Q(z) \in (-\infty,0).
\end{equation}
Now, pick $\theta \in (0,\pi/2)$ and set $S_{\theta} \vcentcolon = \{z \in \H : \abs{\arg(z)-\pi/2} \leq \theta\}$. Choose $L > 0$ sufficiently large so that $\arg(Q(z)) \in (\pi-\theta,\pi+\theta)$ whenever $z \in S_{\theta,L} \vcentcolon = S_{\theta} \cap \{z \in \H : \mathrm{Im}(z) > L\}$. Such a choice is allowed due to \eqref{eq:orthogonal 1}. If $z \in S_{\theta,L}$, we claim that $\{\phi_t(z) : t \ge 0\} \subset S_{\theta,L}$. If this was not true, then we could find $T > 0$ so that $\abs{\arg(\phi_T(z))-\pi/2} = \theta$, since by Julia's Lemma $\mathrm{Im}\phi_t(z)\ge\mathrm{Im}z>L$, for all $t\ge0$. Assume without loss of generality that $\arg(\phi_T(z)) = \pi/2-\theta$ (the argument for the other case is almost identical). Then, via \eqref{eq:inf-gen},
\begin{align*}
\arg\left(\left.\frac{\partial\phi_t(z)}{\partial t}\right|_{t=T}\right) & = \arg(G(\phi_T(z))) = \arg(Q(\phi_T(z))) - \arg(\phi_T(z)) \\
& \geq \pi - \theta-\left(\dfrac{\pi}{2}-\theta\right) = \dfrac{\pi}{2}.
\end{align*}
In other words, if the orbit touches the boundary of $S_{\theta,L}$, then it immediately goes back to its interior, and so the claim is proved.

Accordingly, we conclude that for every $\theta \in (0,\pi/2)$ there exists $z = z(\theta) \in \H$ such that
\begin{equation}
\label{eq:orthogonal 2}
\dfrac{\pi}{2}-\theta \leq \liminf_{t \to +\infty}\arg(\phi_t(z)) \leq \limsup_{t \to +\infty}\arg(\phi_t(z)) \leq \dfrac{\pi}{2}+\theta.  
\end{equation}
However, such limits are independent of $z$ \cite[Theorem 17.5.1]{BCDM}. Then, letting $\theta \to 0^+$ in \eqref{eq:orthogonal 2}, the result is derived.
\end{proof}
\end{lemma}

We are now in a position to prove the main result of this section.

\begin{theorem}
\label{thm:0HS-Herglotz}
Let $(\phi_t)$ be a parabolic semigroup of zero hyperbolic step in $\H$ with Denjoy--Wolff point infinity. Let $G \vcentcolon \H \to \H \cup \R$ be the infinitesimal generator of $(\phi_t)$, represented by the triplet $(\alpha,\beta,\mu)$. The following are equivalent:
\begin{enumerate}[\hspace{0.5cm}\rm(a)]
\item $(\phi_t)$ is of extremal rate;
\item $\displaystyle\limsup_{t \to +\infty}\dfrac{\abs{\phi_t(z)}}{\sqrt{t}} < +\infty$ for some (and hence all) $z \in \H$;
\item $\displaystyle\alpha = 0, \quad \int_{\R}s^2d\mu(s) < +\infty, \quad\text{and}\quad \beta = \int_{\R}sd\mu(s)$.
\end{enumerate}

Moreover, if any of the conditions above are met, we have
\begin{equation}
\label{eq:0HS-Herglotz limit}
\lim_{t \to +\infty}\dfrac{\phi_t(z)}{\sqrt{t}} = i\sqrt{2\int_{\R}(1+s^2)d\mu(s)}, \quad \textup{for all }z\in\H.
\end{equation}
\begin{proof}
First of all, since $(\phi_t)$ is parabolic with Denjoy--Wolff point infinity, by Lemma \ref{lemma:infinitesimal-generator}, we deduce that $\alpha = 0$ regardless of the rate.

Clearly, (a) implies (b). So, we proceed with the two main implications in this proof.

Assume that (b) holds. Then, fix $z \in \H$ with $\limsup_{t \to +\infty}(\abs{\phi_t(z)}/\sqrt{t}) < +\infty$. For $t \ge 0$, set $x_t = \mathrm{Re}(\phi_t(z))$ and $y_t = \mathrm{Im}(\phi_t(z))$. By the definition of the infinitesimal generator and its Herglotz representation, we have
\begin{equation}\label{eq:0HS-Herglotz 1}
\frac{\partial y_t}{\partial t} = \mathrm{Im}(G(\phi_t(z))) = \int_{\R}\dfrac{(1+s^2)y_t}{(s-x_t)^2+y_t^2}d\mu(s).
\end{equation}
By \cite[Proposition 3.8 and Remark 7.5]{Bracci}, the orbit of $z$ converges non-tangentially to infinity. This means that there exists some $a>0$ such that $|x_t| \le a y_t$ for all sufficiently large $t$. A simple calculation (see \cite[Lemma 4.2]{CZZ}) shows that this implies the existence of some $c>0$ such that $(s-x_t)^2+y_t^2 \le c(s^2+y_t^2)$ for all $s \in \R$ and sufficiently large $t$. Thus we see that
$$y_t\frac{\partial y_t}{\partial t} \geq \frac{1}{c}\int_{\R}\dfrac{(1+s^2)y_t^2}{s^2+y_t^2}d\mu(s),$$
for all sufficiently large $t\ge0$. In particular, by Fatou's Lemma, since $\lim_{t \to +\infty}y_t = +\infty$, we have that
\begin{equation}\label{eq:0HS-Herglotz 2}
\liminf_{t \to +\infty}\left(y_t\frac{\partial y_t}{\partial t}\right) \geq \frac{1}{c}\int_{\R}(1+s^2)d\mu(s).
\end{equation}
Aiming towards a contradiction, assume $\int_{\R}s^2d\mu(s) = + \infty$. Then, \eqref{eq:0HS-Herglotz 2} implies that for every $A > 0$ we may find $T > 0$ large enough so that $y_t\frac{\partial y_t}{\partial t} = \frac{\partial (y_t^2/2)}{\partial t} \geq A$, for all $t \geq T$. In particular, for $t \geq T$,
$$\dfrac{y_t^2}{2} = \dfrac{y_T^2}{2}+\int_T^t \frac{1}{2}\dfrac{\partial (y_s^2)}{\partial s}ds \geq \dfrac{y_T^2}{2} + A(t-T),$$
and as a result,
$$\dfrac{y_t}{\sqrt{t}} \geq \sqrt{\dfrac{y_T^2}{t} + \dfrac{2A(t-T)}{t}}.$$ By the arbitrariness of $A > 0$, we conclude that
$$\liminf_{t \to +\infty}\dfrac{y_t}{\sqrt{t}} = +\infty.$$
Due to (b) this cannot hold, and we are led to a contradiction. Hence $\int_{\R}s^2d\mu(s) < +\infty$. Using this fact, for $T > 0$ we may write after some recombinations
\begin{align}\label{eq:0HS-Herglotz 3}
\notag x_T & = x_0 + \int_0^T\mathrm{Re}(G(\phi_t(z)))dt \\
& = x_0 + T\left(\beta - \int_{\R}sd\mu(s)\right)+\int_0^T\int_{\R}\dfrac{(1+s^2)(s-x_t)}{(s-x_t)^2+y_t^2}d\mu(s)dt.
\end{align}
Using Lebesgue's Dominated Convergence Theorem, we see that
$$\lim_{t \to + \infty}\int_{\R}\dfrac{(1+s^2)(s-x_t)}{(s-x_t)^2+y_t^2}d\mu(s) = 0.$$
Consequently, differentiating with respect to $T$ in \eqref{eq:0HS-Herglotz 3} and taking limits, we have $\lim_{T \to +\infty}\frac{\partial x_T}{\partial T} = \beta - \int_{\R}sd\mu(s)$. If $\beta - \int_{\R}sd\mu(s) > 0$ (and likewise if it is negative), it is easy to see that
$$\lim_{t \to +\infty}x_t = +\infty \quad\textup{and}\quad \lim_{t \to +\infty}\dfrac{x_t}{t} = \beta - \int_{\R}sd\mu(s) > 0.$$
Therefore, we obtain that necessarily $\lim_{t \to +\infty}(x_t/\sqrt{t}) = +\infty$. By (b) this cannot hold, and so we have that $\beta = \int_{\R}sd\mu(s)$.

Let us now assume the conditions on $(\alpha,\beta,\mu)$ given in (c) hold. By Lemma \ref{lemma:orthogonal-convergence} the orbits of $(\phi_t)$ converge to infinity orthogonally. This signifies $\lim_{t\to+\infty}(\phi_t(z)/y_t)=i$, for some (and hence all) $z\in\H$, where $y_t=\mathrm{Im}(\phi_t(z))$. For this reason, it is enough to prove that $\lim_{t \to +\infty}(y_t/\sqrt{t}) \in (0,+\infty)$. As before, using Lebesgue's Dominated Convergence Theorem, we see through \eqref{eq:0HS-Herglotz 1} that
$$\lim_{t \to +\infty}\left(y_t\dfrac{\partial y_t}{\partial t}\right) = \lim_{t \to +\infty}\int_{\R}\dfrac{(1+s^2)y_t^2}{(s-x_t)^2+y_t^2}d\mu(s) = \int_{\R}(1+s^2)d\mu(s).$$
By L'H\^opital's rule,
$$\lim_{t \to +\infty}\dfrac{y_t^2}{2t} = \lim_{t \to +\infty} \left(y_t \frac{\partial y_t}{\partial t}\right) = \int_{\R}(1+s^2)d\mu(s).$$
Then,
$$\lim_{t \to +\infty}\dfrac{y_t}{\sqrt{t}} = \sqrt{2\int_{\R}(1+s^2)d\mu(s)}\in(0,+\infty),$$
which proves the desired implication and, in turn, relation \eqref{eq:0HS-Herglotz limit}.
\end{proof}
\end{theorem}

The latter result provides an important feature that we highlight for further reference:

\begin{corollary}
\label{corollary:0HS-Herglotz}
Let $(\phi_t)$ be a parabolic semigroup of zero hyperbolic step in $\H$ with Denjoy--Wolff point infinity. If $(\phi_t)$ is of extremal rate, then all orbits converge orthogonally.
\begin{proof}
This is just a consequence of \eqref{eq:0HS-Herglotz limit} (also of Lemma \ref{lemma:orthogonal-convergence} and the equivalence between conditions (a) and (c) in Theorem \ref{thm:0HS-Herglotz}).
\end{proof}
\end{corollary}

Having Theorem \ref{thm:0HS-Herglotz} at our disposal, we may also characterize parabolic semigroups of zero hyperbolic step attaining the extremal rate based on the behaviour of their infinitesimal generator near the Denjoy--Wolff point. As a matter of fact, we may prove the following handy result:
\begin{theorem}
\label{thm:0HS-properties}
Let $(\phi_t)$ be a parabolic semigroup of zero hyperbolic step in $\H$ with Denjoy--Wolff point infinity. Let $G \vcentcolon \H \to \H$ be its infinitesimal generator. Then, the semigroup $(\phi_t)$ is of extremal rate if and only if $\angle\lim_{z \to \infty}(zG(z)) \in (-\infty,0)$.

Moreover, in such a case, we have
\begin{equation}\label{eq:asymptotic behaviour}
\lim_{t \to +\infty}\dfrac{|\phi_t(z)|}{\sqrt{t}} = \sqrt{-2\angle\lim_{z \to \infty}(zG(z))}.
\end{equation}
\begin{proof}
The result follows from Theorem \ref{thm:0HS-Herglotz} by applying the techniques found in the proof of \cite[Proposition 3.1]{HS}. For this reason, we will only provide a detailed sketch of the proof.

Suppose that $G$ is represented by the triplet $(\alpha,\beta,\mu)$. Using Theorem \ref{thm:0HS-Herglotz}, $(\phi_t)$ is of extremal rate if and only if $\alpha = 0$, $\int_{\R}s^2d\mu(s) < +\infty$, and $\beta = \int_{\R}sd\mu(s)$.

For the direct implication, assume that the conditions on the triplet $(\alpha,\beta,\mu)$ are satisfied. Notice that $\mu$ is not the null measure (if it was, then $G \equiv 0$, and so $\phi_t = \mathrm{id}_{\H}$ for all $t > 0$). Then, as in \eqref{eq:angle-zGz}, we have that
$$\angle\lim_{z \to \infty}(zG(z)) = \angle\lim_{z \to \infty}\int_{\R}\dfrac{z(1+s^2)}{s-z}d\mu(s) = -\int_{\R}(1+s^2)d\mu(s) \in (-\infty,0).$$
In particular, the formula \eqref{eq:asymptotic behaviour} is a direct consequence of \eqref{eq:0HS-Herglotz limit}.

For the reverse implication, assume $L \vcentcolon = \angle\lim_{z \to \infty}(zG(z)) \in (-\infty,0)$. Since $(\phi_t)$ is parabolic, $\alpha = 0$ (see Lemma \ref{lemma:infinitesimal-generator}). In particular, we have that
\begin{align*}
\lim_{y \to +\infty}\mathrm{Re}(iyG(iy)) & = \lim_{y \to +\infty}(-y\mathrm{Im}(G(iy))) \\
& = -\lim_{y \to +\infty}\int_{\R}\dfrac{(1+s^2)y^2}{s^2+y^2}d\mu(s) = L.
\end{align*}
Using Fatou's Lemma, we have that
$$\int_{\R}(1+s^2)d\mu(s) \leq \lim_{y \to +\infty}\int_{\R}\dfrac{(1+s^2)y^2}{s^2+y^2}d\mu(s) = -L < +\infty,$$
which means $\int_{\R}s^2d\mu(s) < +\infty$.

It remains to show that $\beta = \int_{\R}sd\mu(s)$. To do so, since $\alpha = 0$ and $\int_{\R}s^2d\mu(s) < +\infty$, we notice that $G$ can be written as
$$G(z) = \beta - \int_{\R}sd\mu(s) + Q(z), \qquad Q(z) = \int_{\R}\dfrac{1+s^2}{s-z}d\mu(s),$$
where, as argued on the first part of the proof, we have $\angle\lim_{z \to \infty}(zQ(z)) \in (-\infty,0)$. Therefore, we automatically conclude that
$$\angle\lim_{z \to \infty}\left[z\left(\beta-\int_{\R}sd\mu(s)\right)\right] \in \R.$$
Of course, this means that $\beta = \int_{\R}sd\mu(s)$, as needed.
\end{proof}
\end{theorem}

\section{Conformality}
\label{sec:conformality}
A central and well-explored topic in Geometric Function Theory is the study of how a holomorphic map $f \colon \H \to \mathbb{C}$ behaves near a boundary point $\xi \in \mathbb{R} \cup \{\infty\}$. In this part of our work, we will heavily rely on the idea of a map's conformality at a boundary point, a concept that has also been important in the field of Complex Dynamics \cite{BCDM-Rate,CDP-C2,GKMR,GKR}. To be more precise, motivated by Theorems \ref{thm:hyp-conformality} and \ref{thm:phs-conformality}, we will also establish a connection between parabolic semigroups of zero hyperbolic step that have an extremal rate of convergence and the conformality of a modification of their associated Koenigs functions at the Denjoy--Wolff point. 

Given our chosen normalization, we will focus in the case of semigroups whose Denjoy--Wolff point is infinity. Therefore, infinity is a boundary fixed point for the Koenigs functions of such semigroups. Thus, we restrict ourselves to introduce the notion of conformality at infinity for maps satisfying
\begin{equation*}
f(\infty)\vcentcolon = \angle\lim_{z \to \infty}f(z) = \infty.
\end{equation*}
In order to do so, we mostly follow \cite[Section 4.3]{PommerenkeConfMaps}.
\begin{definition}
\label{def:conformality-H-map}
Let $f \vcentcolon \H \to \C$ be a holomorphic map with $f(\infty) = \infty$. We say that $f$ is \textit{conformal at infinity} if
\begin{equation}
\label{eq:conformality-H}
f'(\infty)\vcentcolon = \angle\lim_{z \to \infty}\dfrac{f(z)}{z} \in \C \setminus \{0\}.
\end{equation}
\end{definition}
\begin{remark}
\label{remark:positive-angular-derivative}
{\normalfont (cf. \cite[Theorem 1.7.3]{BCDM})}
Let $f \vcentcolon \H \to \H$ be a holomorphic map with $f(\infty) = \infty$. By the Julia--Wolff--Carath\'eodory Theorem, $f'(\infty) \in (0,+\infty)$. 
\end{remark}

Furthermore, a notion of conformality may be also defined for domains in place of mappings. This can be achieved through the use of conformal mappings.
\begin{definition}
Let $\Omega \subset \C$ be an unbounded simply connected domain. We say that $\Omega$ is \textit{conformal at infinity} if there exists a conformal mapping $f \vcentcolon \H \to \C$ with $\Omega = f(\H)$ and $f(\infty) = \infty$ which is conformal at infinity.
\end{definition}
\begin{remark}
If $\Omega \subset \C$ is conformal at infinity, then every conformal mapping $f \vcentcolon \H \to \Omega$ with $f(\infty) = \infty$ is conformal at infinity. Hence, conformality at infinity is an intrinsic geometric property of the domain.
\end{remark}

Before we present and prove our next batch of results, we will first state some well-known facts about conformality.
\begin{proposition}
\label{prop:angular-derivative}
{\normalfont \cite[Proposition 4.7]{PommerenkeConfMaps}}
Let $f \vcentcolon \H \to \C$ be a holomorphic map with $f(\infty) = \infty$. Then, $f$ is conformal at infinity if and only if
$$\angle\lim_{z \to \infty}f'(z) \in \C \setminus \{0\}.$$
Moreover, in such a case, we have
$$f'(\infty)=\angle\lim_{z \to \infty}\dfrac{f(z)}{z} = \angle\lim_{z \to \infty}f'(z).$$
\end{proposition}

\begin{proposition}
\label{prop:inverse}
{\normalfont \cite[p. 175]{GarnettMarshall}}
Let $f \vcentcolon \H \to \Omega \subset \H$ be a conformal mapping with $f(\infty) = \infty$. Assume that $f$ is conformal at infinity. Then, the following facts hold:
\begin{enumerate}[\hspace{0.5cm}\normalfont(i)]
\item For every $\theta \in (0,\pi/2)$ there exists $R > 0$ so that $\{z \in A_{\theta} : \abs{z} > R\} \subset \Omega$, where $A_{\theta} \vcentcolon = \{z \in \C : \abs{\arg(z)-\pi/2} < \theta\}$.
\item There exists $L \in \C \setminus \{0\}$ such that for every $\theta \in (0,\pi/2)$ we have that
$$\lim_{\substack{z \to \infty \\ z \in A_{\theta}\cap\Omega}}\dfrac{f^{-1}(z)}{z} = L.$$
\end{enumerate}
\end{proposition}

\begin{remark}
\label{remark:inverse}
Let $\Omega \subset \H$ be a simply connected domain, and let $f \vcentcolon \H \to \Omega$ be a conformal mapping with $f(\infty) = \infty$. By Remark \ref{remark:positive-angular-derivative} and Proposition \ref{prop:angular-derivative}, we see that
$$\lim_{\substack{z \to \infty \\ z \in A_{\theta}\cap\Omega}}\dfrac{f^{-1}(z)}{z} = \left(\angle\lim_{z \to \infty}f'(z)\right)^{-1} > 0, \qquad \theta \in (0,\pi/2).$$
\end{remark}

Having settled the necessary background, we proceed to the main body of this section. First, we will notice that conformality at the Denjoy--Wolff point of the actual Koenigs function is not related to the extremality of the rate of convergence of the associated semigroup, at all. This suggests that the analysis of the parabolic semigroups of zero hyperbolic step is more involved than the other cases. To demonstrate this, let us focus on two simple examples:

Consider the semigroup $(\phi_t)$ associated with the Koenigs function $h \vcentcolon \H \to \C$ given by $h(z) = -iz$; i.e. $\phi_t(z) = h^{-1}(h(z)+t)$ for all $z \in \H$ and all $t \geq 0$. It can be quickly confirmed that $(\phi_t)$ is parabolic of zero hyperbolic step. In this case, $h$ (and by extension the unbounded simply connected domain $h(\H)$) is conformal at infinity but the rate of convergence is not extremal. Indeed, straightforward calculations show that $\abs{\phi_t(z)}$ is comparable to $t$ and not to $\sqrt{t}$, as $t \to + \infty$. On the other hand, one may check that $\sqrt{h}$ is well-defined, but not conformal at infinity. Analogously, for $h(z) = -z^2$, it holds that $h$ and $h(\H)$ are not conformal at infinity, even though the rate of convergence for the corresponding semigroup is indeed extremal. This time, $\sqrt{h}$ is well-defined and conformal at infinity.

As a byproduct, we infer that the property of attaining the extremal rate could be characterized in terms of the conformality at infinity of $\sqrt{h}$. Note that, up to some normalization, this function is always well-defined. In fact, the Koenigs domain $\Omega$ of a parabolic semigroup $(\phi_t)$ of zero hyperbolic step is always a proper subset of $\C$ and thus there exists at least one point $c\in\partial\Omega$. But due to Theorem \ref{thm:Koenigs-phs}, the whole half-line $\{c-t\vcentcolon t\ge0\}$ is contained in $\C\setminus\Omega$. Therefore, $\Omega-c\subset \C\setminus(-\infty,0]$ and hence $\sqrt{h-c}$ is well-defined. Guided by this observation, in order to build a level of intuition, we derive some helpful results. The first of them relates the conformality of $\sqrt{h}$ with a notion of conformality with respect to the ``Koebe-like'' domain $K \vcentcolon = \C \setminus (-\infty,0]$, which will frequently appear throughout the paper from now on.

\begin{lemma}
\label{lemma:sqrt}
Let $\Omega \subset K$ be a simply connected domain. The following are equivalent:
\begin{enumerate}[\hspace{0.5cm}\normalfont(a)]
\item There exists a conformal mapping $f \vcentcolon \H \to \Omega$ with $f(\infty) = \infty$ such that $\sqrt{f}$ is conformal at infinity, where we take the branch of the square root that maps $K$ conformally onto the right half-plane.
\item There exists a conformal mapping $g \vcentcolon K \to \Omega$ with $g(\infty) = \infty$ such that
$$\angle\lim_{z \to \infty}\dfrac{g(z)}{z} \in \C \setminus \{0\}.$$
\end{enumerate}
\begin{proof}
It is enough to realize that $g \vcentcolon K \to \Omega$ is a conformal mapping if and only if $f \vcentcolon \H \to \Omega$, given by $f(z) = g(-z^2)$ for $z \in \H$ (or equivalently $g(w)=f(i\sqrt{w})$ for $w\in K$), is also a conformal mapping. Therefore,
$$\angle\lim_{z \to \infty}\dfrac{\sqrt{f(z)}}{z} = \angle\lim_{z \to \infty}\dfrac{\sqrt{g(-z^2)}}{z} = i\angle\lim_{z \to \infty}\sqrt{\dfrac{g(-z^2)}{-z^2}} = i\angle\lim_{w \to \infty}\sqrt{\dfrac{g(w)}{w}},$$
where the angular limits with respect to $z$ are taken in $\H$ and the ones with respect to $w$ are taken in $K$.
\end{proof}
\end{lemma}
Note that conformality is usually studied with respect to various domains. A classical example is the horizontal strip $\mathbb{S}\vcentcolon= \{z\in\C\vcentcolon \mathrm{Im}(z)\in(0,\pi)\}$; see for instance the so-called strip normalization in \cite[p. 176-177]{GarnettMarshall}. {In particular, our second result is obtained by means of the conformal invariance of the hyperbolic distance via a remarkable result of Betsakos and Karamanlis \cite[Theorem 1]{BK-Conformality} concerning the conformality of a domain at $+\infty$, where $+\infty$ denotes the prime end of $\mathbb{S}$ with impression $\infty$ defined by chains of crosscuts with increasing real parts.

\begin{lemma}
\label{lemma:BetsakosKaramanlisK}
Let $\Omega \subset K$ be a simply connected domain containing $(0,+\infty)$. Then, there exists a conformal mapping $f \vcentcolon \H \to \Omega$ with $f(\infty) = \infty$ such that $\sqrt{f}$ is conformal at infinity if and only if the following two conditions hold:
\begin{enumerate}
\item[\hspace{0.5cm}\condI] For every $\alpha \in (0,\pi)$ there exists $R > 0$ such that $\{z \in \C: \abs{z} > R, \, \abs{\arg(z)} < \alpha\} \subset \Omega.$
\item[\hspace{0.5cm}\condII] $d_{\Omega}(a,b) - d_{K}(a,b) \to 0$ as $a,b \to +\infty$, $a,b>0$.
\end{enumerate}
Assuming Condition {\condI} holds, Condition {\condII} may be replaced by any of the following equivalent ones
\begin{enumerate}
\item[\hspace{0.5cm}\condIIp] There exists $z \in \C$ such that $d_{\Omega}(z+a,z+b) - d_{K}(z+a,z+b) \to 0$ as $a,b \to +\infty$, $a,b > 0$.
\item[\hspace{0.5cm}\condIIpp] For all $z \in \C$, $d_{\Omega}(z+a,z+b) - d_{K}(z+a,z+b) \to 0$ as $a,b \to +\infty$, $a,b > 0$.
\end{enumerate}
\begin{proof}
Set $\Omega' \vcentcolon = \{i\sqrt{z} : z \in \Omega\} \subset \H$ and $\Omega''\vcentcolon = \{\log(z) : z \in \Omega'\} \subset \mathbb{S}$. Notice that there exists a conformal mapping $f \vcentcolon \H \to \Omega$ with $f(\infty) = \infty$ such that $\sqrt{f}$ is conformal at infinity if and only if there exists a conformal mapping $h \vcentcolon \H \to \Omega'$ with $h(\infty) = \infty$ (i.e. $h(z) = i\sqrt{f(z)}$, $z \in \H$) which is conformal at infinity. Using the so-called strip normalization (see \cite[pp. 176-177]{GarnettMarshall}), this is also equivalent to the existence of a conformal mapping $g \vcentcolon \mathbb{S} \to \Omega''$ with $g(+\infty) = +\infty$ which is conformal at $+\infty$ (see \cite[Section 2.1]{BK-Conformality} for more details). Using \cite[Theorem 1]{BK-Conformality} and a standard argument of conformal invariance, the equivalence with the given conditions is established.

Under the additional assumption that Condition {\condI} holds, we will now prove the equivalence of Conditions {\condII}, {\condIIp}, and {\condIIpp}. It is obvious that Condition {\condIIpp} implies both Conditions {\condII} and {\condIIp}, and that Condition {\condII} implies Condition {\condIIp}. Thus, we only need to prove that Condition {\condIIp} implies Condition {\condIIpp}. So suppose that indeed
\begin{equation}
\label{eq:betsakos karamanlis 1}
\lim\limits_{a,b>0,\; a,b\to+\infty}\left(d_\Omega(z+a,z+b)-d_K(z+a,z+b)\right)=0.
\end{equation}
holds for some $z \in \C$. Fix $w\in\C$. We will prove that \eqref{eq:betsakos karamanlis 1} holds with $w$ in place of $z$ too. Due to Condition {\condI}, all the horizontal half-lines stretching to infinity in the positive direction (that is with constant imaginary part and increasing real part) are eventually contained in $\Omega$. Therefore, there exists $t_0\ge0$ such that both $z+t$ and $w+t$ belong to $\Omega$, for all $t>t_0$, which signifies that the desired asymptotic behaviour is well-defined. Given that $\Omega\subset K$, the domain monotonicity property of the hyperbolic distance in \eqref{eq:hyperbolic monotonicity} dictates that 
\begin{equation}
\label{eq:betsakos karamanlis 2}
d_\Omega(w+a,w+b)-d_K(w+a,w+b)\ge0,\quad \textup{for all } a,b>t_0.
\end{equation}
On the other hand, by the triangle inequality we obtain
\begin{align}
\label{eq:betsakos karamanlis 3}
d_\Omega(w+a,w+b) -d_K(w+a,w+b) & \leq d_\Omega(z+a,z+b)-d_K(z+a,z+b) \notag \\
& +d_\Omega(z+a,w+a) + d_\Omega(z+b,w+b) \notag \\
& +d_K(z+a,w+a)+d_K(z+b,w+b).
\end{align}
For the desired outcome, we only need to examine the hyperbolic quantity $d_\Omega(z+t,w+t)$, for $t>t_0$. We claim that
$$\lim\limits_{t\to+\infty}d_\Omega(z+t,w+t)=0.$$
Because of Condition {\condI}, we understand that there exists some $R>0$ such that $U\vcentcolon =\{w\in\C:\mathrm{Re}w>R\}\subset\Omega$. Clearly, there exists some $t_1\ge t_0$ such that $z+t$ and $w+t$ lie in $U$ for all $t>t_1$. Using a function to map $U$ conformally onto $\H$, formula \eqref{eq:hypdistance in H}, the conformal invariance of the hyperbolic distance and its domain monotonicity property, we have
\begin{align*}
\lim_{t\to+\infty}d_K(z+t,w+t) & \leq \lim_{t\to+\infty}d_\Omega(z+t,w+t) \leq \lim_{t\to+\infty}d_U(z+t,w+t) \\
& = \lim_{t\to+\infty}d_\H(i(z+t-R),i(w+t-R))\\
& = \dfrac{1}{2}\lim_{t\to+\infty}\log\dfrac{|w+\overline{z}-2R+2t|+|w-z|}{|w+\overline{z}-2R+2t|-|w-z|}=0,
\end{align*}
and our claim is proved. Applying the claim with $a$ and $b$ instead of $t$, we understand from \eqref{eq:betsakos karamanlis 1}, \eqref{eq:betsakos karamanlis 2} and \eqref{eq:betsakos karamanlis 3} that
$$\lim\limits_{a,b>0\;a,b\to+\infty}\left(d_\Omega(w+a,w+b)-d_K(w+a,w+b)\right)=0.$$
The choice of $w\in\C$ was arbitrary and hence Condition {\condIIpp} is satisfied.
\end{proof}
\end{lemma}

For the special case of Koenigs domains, we notice that one of the conditions in the previous lemma is actually superfluous. Recall that a domain $\Omega\subsetneq\C$ is defined as a Koenigs domain provided $\Omega+t\subset\Omega$, for all $t\ge0$. 
\begin{lemma}
\label{lemma:conditions}
Let $\Omega \subset K$ be a Koenigs domain containing $(0,+\infty)$ and satisfying Condition {\condII}. Then, Condition {\condI} also holds.
\begin{proof}
Fix $\alpha \in (0,\pi)$, and assume that Condition {\condI} does not hold. Since $\Omega \subset K$ is a Koenigs domain, this means that
$$\Omega \subset K \setminus \left(\bigcup_{n \in \N} L_n\right),$$
where $L_n \vcentcolon = \{w_n-t : t \geq 0\}$ for some $w_n \in \C$ with $\abs{\arg(w_n)} = \alpha$ and $\lim_{n \to +\infty}w_n = \infty$. We will assume that $\arg(w_n) = \alpha$ for all $n \in \N$ (otherwise the proof is still valid albeit with slight modifications). Moreover, we choose to reorder the points $\{w_n\}$ so that $\abs{w_n} < \abs{w_{n+1}}$ for all $n \in \N$.

We need to prove that Condition {\condII} does not hold. In order to do so, fix $x,y > 0$. For each $n\in\N$ set $K_n \vcentcolon = K \setminus L_n$. Using the monotonicity property of the hyperbolic distance in \eqref{eq:hyperbolic monotonicity}, we see that
\begin{equation}\label{eq:superfluous 1}
d_{\Omega}(x,y) \geq d_{K_1}(x,y).
\end{equation}
For $n \in \N$, $n > 1$, set $c_n \vcentcolon = \abs{w_n/w_1} > 1$. Notice that, by construction, $c_nK_1 = K_n$. Therefore, using also the conformal invariance of the hyperbolic distance in \eqref{eq:hyperbolic invariance}, we have that
\begin{equation}
\label{eq:superfluous 2}
d_{\Omega}(c_nx,c_ny) \geq d_{K_n}(c_nx,c_ny) = d_{K_1}(x,y).
\end{equation}
Also, note that $d_K(cx,cy) = d_K(x,y)$ for all $c > 0$, since $K$ is invariant under homotheties. All in all, we conclude that
$$\liminf_{a,b>0,\;a,b \to +\infty}\left(d_{\Omega}(a,b)-d_K(a,b)\right) \ge d_{K_1}(x,y)-d_K(x,y) > 0.$$
\end{proof}
\end{lemma}

All these considerations lead us to the main result in this section, where we use the conformality of the square root of the Koenigs function $h$ of a semigroup. As we mentioned, we can always find $c \in \partial h(\H)$ such that $h(\H) -c \subseteq K$. In such a case, $\sqrt{h-c}$ is a well-defined holomorphic map. Moreover, since Koenigs functions are defined up to translation, without harming generality we may always suppose that $c = 0$, as we do in the next statement.

\begin{theorem}
\label{thm:0HS}
Let $(\phi_t)$ be a parabolic semigroup of zero hyperbolic step in $\H$ with Denjoy--Wolff point infinity, and let $h \vcentcolon \H \to \C$ be a Koenigs function of $(\phi_t)$ satisfying $(0,+\infty) \subset h(\H) \subseteq \C \setminus (-\infty,0]$ (in particular, $\sqrt{h}$ is well-defined). Then, the semigroup $(\phi_t)$ is of extremal rate if and only if $\sqrt{h}$ is conformal at infinity.

Moreover, in such a case, we have that
\begin{equation}\label{eq:conformality formula}
  \lim_{t \to +\infty}\dfrac{\phi_t(z)}{\sqrt{t}} = i \left(\angle\lim_{w \to \infty}\dfrac{\sqrt{h(w)}}{w}\right)^{-1}, \qquad \text{for all } z \in \H.\end{equation}
\begin{proof}
First, assume that $\sqrt{h}$ is conformal at infinity. For the sake of simplicity, denote $\widetilde{h} = i\sqrt{h} \vcentcolon \H \to \H$ (notice that $\Omega \vcentcolon = h(\D) \subset K = \C \setminus (-\infty,0]$). Proceeding as in the proof of Lemma \ref{lemma:sqrt} we can see that $\widetilde{h}$ is conformal at infinity. Using Proposition \ref{prop:inverse} and Remark \ref{remark:inverse}, there exists $L > 0$ such that for every $\theta \in (0,\pi/2)$ we have that
\begin{equation}
\label{eq:conformal-L}
\lim_{\substack{z \to \infty \\ z \in \widetilde{h}(\H) \cap A_{\theta}}}\dfrac{\widetilde{h}^{-1}(z)}{z} = L,
\end{equation}
where $A_{\theta} \vcentcolon = \{z \in \C : \abs{\arg(z)} < \theta\}$.

We also note that the Abel equation $h(\phi_t(z)) = h(z) + t$, for $z \in \H$ and $t \geq 0$, is equivalent to
$$\widetilde{h}(\phi_t(z)) = i\sqrt{h(\phi_t(z))} = i \sqrt{h(z)+t} = i\sqrt{t-\widetilde{h}(z)^2}.$$
Using this, we have that
\begin{align}
\label{eq:lim-calculation}
\lim_{t \to +\infty}\dfrac{\phi_t(z)}{\sqrt{t}} & = \lim_{t \to +\infty}\dfrac{\widetilde{h}^{-1}\left(i\sqrt{t-\widetilde{h}(z)^2}\right)}{\sqrt{t}} \\
& = \lim_{t \to +\infty}\left(\dfrac{\widetilde{h}^{-1}\left(i\sqrt{t-\widetilde{h}(z)^2}\right)}{i\sqrt{t-\widetilde{h}(z)^2}}\dfrac{i\sqrt{t-\widetilde{h}(z)^2}}{\sqrt{t}}\right) = iL, \notag
\end{align}
where we have used \eqref{eq:conformal-L}. That is, the semigroup is of extremal rate. In particular, formula \eqref{eq:conformality formula} of the statement follows from \eqref{eq:conformal-L} and \eqref{eq:lim-calculation}.

For the converse direction, let us assume now that $(\phi_t)$ is of extremal rate. Using Theorem \ref{thm:0HS-Herglotz}, we further see that there exists $L > 0$ so that
\begin{equation}
\label{eq:L}
\lim_{t \to +\infty}\dfrac{\phi_t(z)}{\sqrt{t}} = iL.
\end{equation}
Moreover, by Lemma \ref{lemma:BetsakosKaramanlisK} and Lemma \ref{lemma:conditions}, it is enough to prove that Condition {\condII} holds to deduce that $\sqrt{h}$ is conformal at infinity. To do so, since $(0,+\infty) \subset \Omega$, let $z \in \H$ be such that $h(z) = 1$. If $x > 1$, we see that
$$x = h(z) + x - 1 = h(\phi_{x-1}(z)).$$
Then, for $a,b > 1$, we have that
\begin{align*}
d_{\Omega}(a,b) - d_K(a,b) & = d_{\Omega}(h(\phi_{a-1}(z)),h(\phi_{b-1}(z))) - d_K(a,b) \\
& = d_{\H}(\phi_{a-1}(z),\phi_{b-1}(z)) - d_{\H}(i\sqrt{a},i\sqrt{b}).
\end{align*}

Let $\{a_n\},\{b_n\} \subset (0,+\infty)$ be sequences with $a_n,b_n \to +\infty$, as $n \to +\infty$. It is always possible to find $0 \leq M \leq +\infty$ and subsequences $\{a_{n_k}\},\{b_{n_k}\} \subset (0,+\infty)$ so that $a_{n_k}/b_{n_k} \to M$, as $k \to +\infty$. 

If $M = 1$, using \eqref{eq:hypdistance in H} we have that
\begin{align*}
d_{\H}(i\sqrt{a_{n_k}},i\sqrt{b_{n_k}}) & = \dfrac{1}{2}\log\left[\dfrac{\abs{\sqrt{a_{n_k}}+\sqrt{b_{n_k}}}+\abs{\sqrt{a_{n_k}}-\sqrt{b_{n_k}}}}{\abs{\sqrt{a_{n_k}}+\sqrt{b_{n_k}}}-\abs{\sqrt{a_{n_k}}-\sqrt{b_{n_k}}}}\right] = \\
& = \dfrac{1}{2}\log\left[\dfrac{\abs{1+\sqrt{b_{n_k}}/\sqrt{a_{n_k}}}+\abs{1-\sqrt{b_{n_k}}/\sqrt{a_{n_k}}}}{\abs{1+\sqrt{b_{n_k}}/\sqrt{a_{n_k}}}-\abs{1-\sqrt{b_{n_k}}/\sqrt{a_{n_k}}}}\right] \to 0,
\end{align*}
as $k \to +\infty$. Similarly,
\begin{align*}
& d_{\H}(\phi_{a_{n_k}-1}(z),\phi_{b_{n_k}-1}(z)) \\
& = \dfrac{1}{2}\log\left[\dfrac{\abs{\phi_{a_{n_k}-1}(z)-\overline{\phi_{b_{n_k}-1}(z)}}+\abs{\phi_{a_{n_k}-1}(z)-\phi_{b_{n_k}-1}(z)}}{\abs{\phi_{a_{n_k}-1}(z)-\overline{\phi_{b_{n_k}-1}(z)}}-\abs{\phi_{a_{n_k}-1}(z)-\phi_{b_{n_k}-1}(z)}}\right] \\
& = \dfrac{1}{2}\log\left[\dfrac{\abs{\frac{\phi_{a_{n_k}-1}(z)}{\sqrt{a_{n_k}}}-\frac{\overline{\phi_{b_{n_k}-1}(z)}}{\sqrt{b_{n_k}}}\frac{\sqrt{b_{n_k}}}{\sqrt{a_{n_k}}}}+\abs{\frac{\phi_{a_{n_k}-1}(z)}{\sqrt{a_{n_k}}}-\frac{\phi_{b_{n_k}-1}(z)}{\sqrt{b_{n_k}}}\frac{\sqrt{b_{n_k}}}{\sqrt{a_{n_k}}}}}{\abs{\frac{\phi_{a_{n_k}-1}(z)}{\sqrt{a_{n_k}}}-\frac{\overline{\phi_{b_{n_k}-1}(z)}}{\sqrt{b_{n_k}}}\frac{\sqrt{b_{n_k}}}{\sqrt{a_{n_k}}}}-\abs{\frac{\phi_{a_{n_k}-1}(z)}{\sqrt{a_{n_k}}}-\frac{\phi_{b_{n_k}-1}(z)}{\sqrt{b_{n_k}}}\frac{\sqrt{b_{n_k}}}{\sqrt{a_{n_k}}}}}\right] \\
& \to \dfrac{1}{2}\log\left[\dfrac{\abs{iL+iL}+\abs{iL-iL}}{\abs{iL+iL}-\abs{iL-iL}}\right] = 0,
\end{align*}
as $k \to +\infty$, where we have used \eqref{eq:L} and the fact that $a_{n_k}/b_{n_k} \to M = 1$.

If $M \in (1,+\infty)$, we can proceed similarly. First of all, up to finitely many terms, we may suppose that $a_{n_k} > b_{n_k}$. In that case, via \eqref{eq:hypdistance in H}
$$d_{\H}(i\sqrt{a_{n_k}},i\sqrt{b_{n_k}}) = \dfrac{1}{2}\log\left(\dfrac{\sqrt{a_{n_k}}}{\sqrt{b_{n_k}}}\right).$$
Consequently,
\begin{align*}
& d_{\H}(\phi_{a_{n_k}-1}(z),\phi_{b_{n_k}-1}(z)) - d_{\H}(i\sqrt{a_{n_k}},i\sqrt{b_{n_k}}) \\
& = \dfrac{1}{2}\log\left[\dfrac{\abs{\phi_{a_{n_k}-1}(z)-\overline{\phi_{b_{n_k}-1}(z)}}+\abs{\phi_{a_{n_k}-1}(z)-\phi_{b_{n_k}-1}(z)}}{\abs{\phi_{a_{n_k}-1}(z)-\overline{\phi_{b_{n_k}-1}(z)}}-\abs{\phi_{a_{n_k}-1}(z)-\phi_{b_{n_k}-1}(z)}}\dfrac{\sqrt{b_{n_k}}}{\sqrt{a_{n_k}}}\right] \\
& = \dfrac{1}{2}\log\left[\dfrac{\abs{\frac{\sqrt{a_{n_k}}}{\sqrt{b_{n_k}}}-\frac{\overline{\phi_{b_{n_k}-1}(z)}/\sqrt{b_{n_k}}}{\phi_{a_{n_k}-1}(z)/\sqrt{a_{n_k}}}}+\abs{\frac{\sqrt{a_{n_k}}}{\sqrt{b_{n_k}}}-\frac{\phi_{b_{n_k}-1}(z)/\sqrt{b_{n_k}}}{\phi_{a_{n_k}-1}(z)/\sqrt{a_{n_k}}}}}{\abs{\frac{\sqrt{a_{n_k}}}{\sqrt{b_{n_k}}}-\frac{\overline{\phi_{b_{n_k}-1}(z)}/\sqrt{b_{n_k}}}{\phi_{a_{n_k}-1}(z)/\sqrt{a_{n_k}}}}-\abs{\frac{\sqrt{a_{n_k}}}{\sqrt{b_{n_k}}}-\frac{\phi_{b_{n_k}-1}(z)/\sqrt{b_{n_k}}}{\phi_{a_{n_k}-1}(z)/\sqrt{a_{n_k}}}}}\dfrac{\sqrt{b_{n_k}}}{\sqrt{a_{n_k}}}\right] \\
& \to \dfrac{1}{2}\log\left[\dfrac{\abs{\sqrt{M}+1}+\abs{\sqrt{M}-1}}{\abs{\sqrt{M}+1}-\abs{\sqrt{M}-1}}\dfrac{1}{\sqrt{M}}\right] = 0,
\end{align*}
as $k \to +\infty$, where we have used \eqref{eq:L}.

If $M = +\infty$, further calculations are needed. As before, up to a finite number of terms, we suppose that $a_{n_k} > b_{n_k}$. For $w_1,w_2  \in \H$, we use that
$$\dfrac{\abs{w_1-\overline{w_2}}+\abs{w_1-w_2}}{\abs{w_1-\overline{w_2}}-\abs{w_1-w_2}} = \dfrac{(\abs{w_1-\overline{w_2}}+\abs{w_1-w_2})^2}{4\mathrm{Im}(w_1)\mathrm{Im}(w_2)}.$$
Therefore,
\begin{align*}
& d_{\H}(\phi_{a_{n_k}-1}(z),\phi_{b_{n_k}-1}(z)) - d_{\H}(i\sqrt{a_{n_k}},i\sqrt{b_{n_k}}) \\
& = \dfrac{1}{2}\log\left[\dfrac{\left(\abs{\phi_{a_{n_k}-1}(z)-\overline{\phi_{b_{n_k}-1}(z)}}+\abs{\phi_{a_{n_k}-1}(z)-\phi_{b_{n_k}-1}(z)}\right)^2}{4\mathrm{Im}(\phi_{a_{n_k}-1}(z))\mathrm{Im}(\phi_{b_{n_k}-1}(z))}\dfrac{\sqrt{b_{n_k}}}{\sqrt{a_{n_k}}}\right] \\
& = \dfrac{1}{2}\log\left[\dfrac{\left(\abs{\frac{\phi_{a_{n_k}-1}(z)}{\sqrt{a_{n_k}}}-\frac{\overline{\phi_{b_{n_k}-1}(z)}}{\sqrt{b_{n_k}}}\frac{\sqrt{b_{n_k}}}{\sqrt{a_{n_k}}}}+\abs{\frac{\phi_{a_{n_k}-1}(z)}{\sqrt{a_{n_k}}}-\frac{\phi_{b_{n_k}-1}(z)}{\sqrt{b_{n_k}}}\frac{\sqrt{b_{n_k}}}{\sqrt{a_{n_k}}}}\right)^2}{4\frac{\mathrm{Im}(\phi_{a_{n_k}-1}(z))}{\sqrt{a_{n_k}}}\frac{\mathrm{Im}(\phi_{b_{n_k}-1}(z))}{\sqrt{b_{n_k}}}}\right] \\
& \to \dfrac{1}{2}\log\left[\dfrac{(\abs{iL-0}+\abs{iL-0})^2}{4L^2}\right] = 0,
\end{align*}
as $k \to +\infty$, where we have used \eqref{eq:L}.

Similar techniques show that, for $M \in [0,1)$, we also have that
$$\lim_{k \to \infty}\left(d_{\Omega}(a_{n_k},b_{n_k}) - d_K(a_{n_k},b_{n_k})\right) = 0.$$

Summing up, we have proved that for every sequence $\{a_n\},\{b_n\} \subset (0,+\infty)$ with $a_n,b_n \to +\infty$, as $n \to +\infty$, there exists subsequences with
$$\lim_{k \to \infty}\left(d_{\Omega}(a_{n_k},b_{n_k}) - d_K(a_{n_k},b_{n_k})\right) = 0.$$
Therefore, we see that
$$\lim_{a,b>0,\;a,b \to +\infty}\left(d_{\Omega}(a,b) - d_K(a,b)\right) = 0$$
and Condition {\condII} holds, as desired.
\end{proof}
\end{theorem}

\section{Hyperbolic distance}
\label{sec:hyp-dist}
Let $(\phi_t)$ be a non-elliptic semigroup in $\H$. The quantity $d_{\H}(i,\phi_t(i))$ is called the \textit{total speed} of $(\phi_t)$ and was introduced by Bracci in \cite{Bracci} (he defined it for the setting of the unit disc, but the conformal invariance of the hyperbolic distance allows us to equivalently define it like this in the upper half-plane). In \cite[Proposition 6.2]{Bracci}, Bracci proves that $\liminf_{t\to+\infty}(d_\H(i,\phi_t(i))-\log (t)/4)>-\infty$ for parabolic semigroups of zero hyperbolic step. We will actually characterize the existence of the corresponding limit via the extremal rate.

\begin{theorem}
\label{thm:hyperbolic rate 0HS}
Let $(\phi_t)$ be a parabolic semigroup of zero hyperbolic step in $\H$ with Denjoy--Wolff point infinity. The following facts are equivalent:
\begin{enumerate}[\hspace{0.5cm}\rm(a)]
\item $(\phi_t)$ is of extremal rate.
\item For some (and hence all) $z \in \H$,
$$\limsup\limits_{t\to+\infty}\left(d_\H(i,\phi_t(z))-\frac{1}{4}\log (t)\right) < +\infty.$$
\item For some (and hence all) $z \in \H$,
$$\lim\limits_{t\to+\infty}\left(d_\H(i,\phi_t(z))-\frac{1}{4}\log (t)\right)$$
exists in $\R$.
\end{enumerate}
Moreover, if any of these conditions are met, then
\begin{equation}
\label{eq:theorem on hyperbolic distance}
\lim\limits_{t\to+\infty}\left(d_\H(i,\phi_t(z))-\frac{1}{4}\log (t)\right)=\dfrac{1}{2}\log\left(\lim\limits_{t\to+\infty}\dfrac{|\phi_t(i)|}{\sqrt{t}}\right),
\end{equation}
for all $z\in\H$.
\end{theorem}
\begin{proof}
Clearly, (c) implies (b). For the rest of the proof, we reason as follows. Fix $z\in\H$. In general, executing simple algebraic computations and with the help of \eqref{eq:hypdistance in H}, for $t\ge0$ we have
\begin{align}
\label{eq:hyp 0HS 11}
d_{\H}(i,\phi_t(z)) & = \frac{1}{2}\log\left(\frac{|\phi_t(z)+i|+|\phi_t(z)-i|}{|\phi_t(z)+i|-|\phi_t(z)-i|}\right) \notag \\
& =\frac{1}{2}\log\left(\frac{|\phi_t(z)+i|^2}{y_t}\right)+\frac{1}{2}\log\left(\frac{1}{4}\left(1+\frac{|\phi_t(z)-i|}{|\phi_t(z)+i|}\right)^2\right),
\end{align}
where $y_t \vcentcolon = \mathrm{Im}(\phi_t(z))$. Taking limits as $t\to+\infty$ in \eqref{eq:hyp 0HS 11}, we obtain
\begin{equation*}
\lim_{t \to + \infty}\left[d_{\H}(i,\phi_t(z))-\dfrac{1}{2}\log\left(\dfrac{\abs{\phi_t(z)+i}^2}{y_t}\right)\right] = 0,
\end{equation*}
or equivalently
\begin{equation}
\label{eq:hyp 0HS 2}
\lim\limits_{t\to+\infty}\left[d_\H(i,\phi_t(z))-\frac{1}{2}\log(|\phi_t(z)+i|)-\frac{1}{2}\log\left(\frac{|\phi_t(z)+i|}{y_t}\right)\right]=0.
\end{equation}

Let us prove that (a) implies (c). If $(\phi_t)$ is of extremal rate, then \eqref{eq:0HS-Herglotz limit} implies the existence of some $\ell\in\R$ such that
\begin{equation}
\label{eq:hyp 0HS 0}
\lim\limits_{t\to+\infty}\left(\log(|\phi_t(z)|)-\frac{1}{2}\log(t)\right)=\ell,
\end{equation}
for all $z\in\H$. Applying \eqref{eq:hyp 0HS 0} on \eqref{eq:hyp 0HS 2}, we deduce that for any $z\in\H$
\begin{equation}
\label{eq:hyp 0HS 3}
\lim\limits_{t\to+\infty}\left[d_\H(i,\phi_t(z))-\frac{1}{4}\log(t)-\frac{1}{2}\log\left(\frac{|\phi_t(z)+i|}{y_t}\right)\right]=\frac{\ell}{2}.
\end{equation}
However, when $(\phi_t)$ is of extremal rate, its orbits converge necessarily orthogonally; see Corollary \ref{corollary:0HS-Herglotz}. Therefore, $|\phi_t(z)+i|/y_t$ converges to $1$, as $t \to +\infty$. Returning to \eqref{eq:hyp 0HS 3}, we understand that the limit of $d_\H(i,\phi_t(z))-\log(t)/4$ converges to $\ell/2$ and in particular, its value does not depend on $z$. This also proves \eqref{eq:theorem on hyperbolic distance}.

Finally, we prove that (b) implies (a). To do so, we assume that $(\phi_t)$ is not of extremal rate. Then, by Theorem \ref{thm:0HS-Herglotz}(b) we have
\begin{equation}
\label{eq:hyp 0HS 1}
\limsup\limits_{t\to+\infty}\left(\log(|\phi_t(z)|)-\frac{1}{2}\log(t)\right)=+\infty.
\end{equation}
A combination of \eqref{eq:hyp 0HS 1} and \eqref{eq:hyp 0HS 2} implies that
$$\limsup\limits_{t\to+\infty}\left[d_\H(i,\phi_t(z))-\frac{1}{4}\log(t)-\frac{1}{2}\log\left(\frac{|\phi_t(z)+i|}{y_t}\right)\right]=+\infty.$$
But clearly $\liminf_{t\to+\infty}(|\phi_t(z)+i|/y_t)\ge1$ and hence it is necessary that
$$\limsup\limits_{t\to+\infty}\left(d_\H(i,\phi_t(z))-\frac{1}{4}\log(t)\right)=+\infty.$$
\end{proof}

\begin{remark}
Using \cite[Theorem 7.1 and Theorem 7.4]{CZZ} and translating to the continuous setting, we can see that a hyperbolic semigroup with spectral value $\lambda>0$ is of extremal rate if and only if the limit in Theorem \ref{thm:hyperbolic rate 0HS} exists with $\lambda t/2$ instead of $\log (t)/4$ and similarly a parabolic semigroup of positive hyperbolic step is of extremal rate (or equivalently of finite shift) if and only if the limit exists with $\log t$ instead of $\log (t)/4$. These results provide new information with respect to the asymptotic behaviour of the total speed of non-elliptic semigroups.
\end{remark}

\section{Composition operators}
\label{sec:comp-operators}
Most of the cited references concerning the rate at which orbits of semigroups approach the Denjoy--Wolff point are written in the setting of the unit disc $\D$. As a matter of fact, we may translate some of the introduced notions about the extremal rate of convergence into the framework of the unit disc. One of the advantages of this  task is that it naturally uncovers a new characterization related to the norms of composition operators.

To do so, we proceed as in \cite[Section 8]{CZZ}. It is well-established that the unit disc and the upper half-plane are conformally equivalent via M\"{o}bius transformations. Consequently, the dynamical features we have outlined for $\H$ can be readily adapted to $\D$. Let us consider a non-elliptic semigroup $(\phi_t)$ in $\H$ for which the Denjoy--Wolff point is infinity, and let $\tau \in \partial\D$. We can then construct a conjugated semigroup $(\psi_t)$ in $\D$ by setting $\psi_t = S^{-1} \circ \phi_t \circ S$, $t\ge0$. The map $S \colon \D \to \H$ is the specific M\"obius transformation given by
\begin{equation}
\label{eq:map-S}
S(z) = i\dfrac{\tau+z}{\tau-z}, \quad z \in \D, \quad S^{-1}(w) = \tau\dfrac{w-i}{w+i}, \quad w \in \H.
\end{equation}
For any $z \in \D$ and $t \ge 0$, the dynamics are related by $\psi_t(z) = S^{-1}(\phi_t(w))$, with $w = S(z)$. Since $S^{-1}(\infty) = \tau$, it follows directly that $\lim_{t \to +\infty}\psi_t(z) = \tau$ for all $z \in \D$. Thus, $\tau$ is identified as the Denjoy--Wolff point of $(\psi_t)$. For this reason, we will refer to $(\psi_t)$ as the \textit{conjugation} of $(\phi_t)$ in $\D$ with Denjoy--Wolff point $\tau$.

From the explicit form of the map in \eqref{eq:map-S}, a straightforward calculation reveals that 
\begin{equation}
\label{eq:rate in D}
\lim_{t\to+\infty}(|\psi_t(z)-\tau||\phi_t(w)|)=2, 
\end{equation}
for any $z\in\D$ and its image $w=S(z)$. This reveals a close relation of the rate of convergence of the orbits of both semigroups, which inspires our next result:

\begin{lemma}\label{lm:parabolic D}
Let $(\phi_t)$ be a parabolic semigroup of zero hyperbolic step in $\H$ with Denjoy--Wolff point infinity and $(\psi_t)$ its conjugation in $\D$ with Denjoy--Wolff point $\tau\in\partial\D$. Then, the following are equivalent:
\begin{enumerate}
\item[\textup{(a)}] $(\phi_t)$ is of extremal rate.
\item[\textup{(b)}] $\lim_{t\to+\infty}(\sqrt{t}\lvert \psi_t(z)-\tau\rvert)$ exists in $(0,+\infty)$ for some (and hence all) $z\in\D$.
\item[\textup{(c)}] $\lim_{t\to+\infty}[\sqrt{t}(1-\lvert \psi_t(z)\rvert)]$ exists in $(0,+\infty)$ for some (and hence all) $z\in\D$.
\end{enumerate}
\end{lemma}
\begin{proof}
The equivalence of statements (a) and (b) is derived at once from Lemma \ref{lemma:rate-0hs}, Definition \ref{def:extremal 0HS} and \eqref{eq:rate in D}. Next, suppose that $(\phi_t)$ is of extremal rate. By Corollary \ref{corollary:0HS-Herglotz}, all the orbits of $(\phi_t)$ diverge to infinity orthogonally. Since M\"{o}bius transformations preserve angles, every orbit of $(\psi_t)$ converges to $\tau$ orthogonally as well. Therefore, $\lim_{t\to+\infty}[\lvert \psi_t(z)-\tau \rvert/(1-\lvert \psi_t(z)\rvert)]=1$, for all $z\in\D$. Since (a) and (b) are equivalent, clearly (a) implies (c).

All that remains is to prove that (c) implies (a). Towards this aim, fix $z \in \D$ and set $w = S(z) \in \H$. Utilizing \eqref{eq:rate in D} and the triangle inequality, for some $C > 0$ and all $t > 0$ we have that
$$\dfrac{\abs{\phi_t(w)}}{\sqrt{t}} \leq \dfrac{C}{\sqrt{t}\abs{\psi_t(z)-\tau}} \leq \dfrac{C}{\sqrt{t}(1-\abs{\psi_t(z)})}.$$
Therefore, it follows from (c) that
$$\limsup_{t\to+\infty}\frac{\lvert\phi_t(w)\rvert}{\sqrt{t}}<+\infty,$$
which by Theorem \ref{thm:0HS-Herglotz} implies that $(\phi_t)$ is of extremal rate.

\end{proof}

Using Lemma \ref{lm:parabolic D} we see that the extremality of the rate of convergence is completely determined by the asymptotic behaviour of $1-|\psi_t(z)|$ when working in the unit disc. This quantity is closely related to the norm of the composition operators on classical spaces of analytic functions on $\D$, as we are about to introduce.

Let $X$ be a Banach space of analytic functions in $\D$, and let $g \vcentcolon \D \to \D$ be a holomorphic map. Then, the \textit{composition operator} $C_g$ induced by $g$ and acting on $X$ is given by $C_g(f) = f \circ g$, for $f \in X$. Clearly, given a semigroup $(\psi_t)$ of holomorphic functions in $\D$, for each term $\psi_t$, $t\ge0$, belonging in the semigroup, we are allowed to define the composition operator $C_t\vcentcolon =C_{\psi_t}$ acting on the Banach space $X$. It can be readily checked that the family of composition operators $(C_t)$, $t\ge0$, satisfies:
\begin{enumerate}
\item[(i)] $C_0=I$, the identity operator on $X$;
\item[(ii)] $C_{t+s}=C_s \circ C_t$, for all $t,s\ge0$.
\end{enumerate}
Therefore, $(C_t)$ is a \textit{one-parameter semigroup of composition operators} on $X$. We will say that the semigroup $(\psi_t)$ \textit{induces} the semigroup $(C_t)$, or that $(C_t)$ is \textit{induced} by $(\psi_t)$. The study of semigroups of composition operators in the setting of Banach spaces of analytic functions initiated in \cite{BP}. For a comprehensive description of the main results on semigroups of composition operators acting on classical spaces of analytic functions see the survey \cite{Siskakis-Hardy}. Besides their general interest, semigroups of composition operators have also been investigated in relation to complex dynamics; see \cite{DB-Bergman, DB-Hardy} or the more recent \cite{BGGY,GC-GD, KTZ-Operators}.

A semigroup of composition operators $(C_t)$ acting on $X$ is said to be \textit{strongly continuous} (also found as a $C_0$\textit{-semigroup} in the literature) provided that $C_t(X) \subset X$, for all $t \geq 0$, and $\lim_{t\to0}||C_t(f)-f||_X=0$, for every $f\in X$, where $||\cdot||_X$ denotes the norm of the Banach space $X$. 

In this article we will only deal with the case in which $X$ is one of the classical Hardy or Bergman spaces. Recall that the \textit{Hardy space} $H^p\vcentcolon =H^p(\D)$, $p>0$, stands for the collection of all the holomorphic functions $f\vcentcolon\D\to\C$ such that
$$\sup\limits_{r\in(0,1)}\int_{0}^{2\pi}|f(re^{i\theta})|^p d\theta<+\infty.$$
Reference \cite{Duren-Hp} provides a complete exposition on these spaces. Every composition operator is bounded from $H^p$ into itself, $p \geq 1$; see Littlewood's Subordination Theorem \cite[Theorem 1.7]{Duren-Hp}. In particular, the norm of the operator may be estimated through the following inequality:
\begin{lemma}
{\rm \cite[Corollary 3.7]{Cowen-Maccluer}}
\label{lm:hardy growth}
Let $g \vcentcolon \D \to \D$ be analytic. Then
\begin{equation}
\label{eq:hardy growth}
\left(\frac{1}{1-|g(0)|^2}\right)^{\frac{1}{p}}\le ||C_g||_{H^p} \le \left(\frac{1+|g(0)|}{1-|g(0)|}\right)^{\frac{1}{p}},
\end{equation}
where $||C_g||_{H^p}$ denotes the norm of the operator $C_g$ acting on the Hardy space $H^p$, $p \geq 1$.
\end{lemma}
In addition, given a semigroup $(\psi_t)$ in $\D$, the induced semigroup $(C_t)$ is strongly continuous on each $H^p$, $p \geq 1$; see \cite[Theorem 3.4]{BP}.

The \textit{Bergman space} $A^p=A^p(\D)$, $p>0$, consists of all the holomorphic functions $f\vcentcolon\D\to\C$ satisfying
$$\int_{\D}|f(z)|^p dA(z)<+\infty,$$
where $dA(z)$ denotes the normalized Lebesgue area measure in $\D$. We refer to \cite{Duren-Apa} for a detailed exposition. Composition operators acting on a Bergman space are also bounded due to Littlewood's Subordination Theorem. In fact, the following estimates on the norm of a composition operator acting on a Bergman space are derived from the combination of \cite[Theorem 11.6]{Zhu} and \cite[Theorem 1]{Vukotic}.

\begin{lemma}
\label{lm:bergman growth}
Let $g \vcentcolon \D \to \D$ be analytic. Then
\begin{equation}
\label{eq:bergman growth}
\left(\frac{1}{1-|g(0)|^2}\right)^{\frac{2}{p}} \le ||C_g||_{A^p} \le \left(\frac{1+|g(0)|}{1-|g(0)|}\right)^{\frac{2}{p}},
\end{equation}
where $||C_g||_{A^p}$ denotes the norm of the operator $C_g$ acting on the Bergman space $A^p$, $p\ge1$.
\end{lemma}

As was the case with the Hardy spaces, by \cite[Theorem 1]{Siskakis-Bergman} we know that every semigroup $(\psi_t)$ in $\D$ induces a strongly continuous semigroup $(C_t)$ on each Bergman space $A^p$, $p \geq 1$.

We are now ready to continue with our final results. Lemma \ref{lm:parabolic D} provides the following result for parabolic semigroups of zero hyperbolic step.

\begin{corollary}
\label{thm:norm 0HS}
Let $(\phi_t)$ be a parabolic semigroup of zero hyperbolic step in $\H$ with Denjoy--Wolff point infinity. Suppose that $(\psi_t)$ is a conjugation of $(\phi_t)$ in $\D$ and that it induces $(C_t)$. The following are equivalent:
\begin{enumerate}[\hspace{0.5cm}\rm(a)]
\item $(\phi_t)$ is of extremal rate;
\item There exists $C_1=:C_1((\phi_t)) \ge 1$ such that
$$\dfrac{1}{C_1} \leq \dfrac{||C_t||_{H^p}^p}{\sqrt{t}} \leq C_1, \quad\textup{for all }t\ge1 \textup{ and all }p\ge1;$$
\item There exists $C_2=:C_2((\phi_t)) \ge 1$ such that
$$\dfrac{1}{C_2} \leq \dfrac{||C_t||_{A^p}^p}{t} \leq C_2, \quad\textup{for all }t\ge1 \textup{ and all }p\ge1.$$
\end{enumerate}
\end{corollary}
\begin{proof}
Applying Lemma \ref{lm:hardy growth} we see that
\begin{equation}\label{eq:hardy1}
\frac{1}{2} \leq ||C_{t}||_{H^p}^p(1-|\psi_t(0)|) \leq 2, \quad\textup{for all }t\ge0 \textup{ and all }p\ge1.
\end{equation}
Then, the equivalence between (a) and (b) follows from Lemma \ref{lm:parabolic D}. Arguing similarly, but this time using Lemma \ref{lm:bergman growth}, provides the equivalence between (a) and (c).
\end{proof}

\end{document}